\documentclass[12pt]{amsart}
\usepackage[margin=1.0in]{geometry}
\usepackage{ifxetex,ifluatex}
\newif\ifxetexorluatex
\ifxetex
  \xetexorluatextrue
\else
  \ifluatex
    \xetexorluatextrue
  \else
    \xetexorluatexfalse
  \fi
\fi

\ifxetexorluatex
  \usepackage{fontspec}
\else
  \usepackage[T1]{fontenc}
  \usepackage[utf8]{inputenc}
  \usepackage{lmodern}
\fi
\usepackage{hyperref, amsmath, amssymb, mathtools}
\usepackage{bookmark}
\usepackage{amsthm}
\usepackage{amsfonts}
\usepackage{amsmath}
\usepackage{comment}
\usepackage{chngcntr}
\usepackage{youngtab}
\newtheorem{thm}{Theorem}
\newtheorem{lemma}[thm]{Lemma}

\theoremstyle{definition}

\newtheorem*{remark}{Remark}
\numberwithin{thm}{section}
\numberwithin{equation}{section}
\DeclareMathOperator{\lcm}{lcm}
\title{Lacunarity of Han-Nekrasov-Okounkov $q$-series }

\author[K. Gallagher]{Katherine Gallagher}
\address{Department of Mathematics, University of Notre Dame, Notre Dame, IN  46556}
\email{kgalla17@nd.edu}

\author[L. Li]{Lucia Li}
\address{Department of Mathematics, Wellesley College, Wellesley, MA 02481}
\email{lucia.li@wellesley.edu}

\author[K. Vassilev]{Katja Vassilev}
\address{Department of Mathematics, Princeton University, Princeton, NJ 08544}
\email{kdv@princeton.edu}

\begin{document}
\maketitle
\begin{abstract}
A power series is called lacunary if ``almost all'' of its coefficients are zero. Integer partitions have motivated the classification of lacunary specializations of Han's extension of the Nekrasov-Okounkov formula. More precisely, we consider the modular forms \[F_{a,b,c}(z) \coloneqq  \frac{\eta(24az)^a \eta(24acz)^{b-a}}{\eta(24z)},\] defined in terms of the Dedekind $\eta$-function, for integers $a,c \geq 1$ where $b \geq 1$ is odd throughout. Serre \cite{Serre} determined the lacunarity of the series when $a = c = 1$. Later, Clader, Kemper, and Wage \cite{REU} extended this result by allowing $a$ to be general, and completely classified the $F_{a,b,1}(z)$ which are lacunary. Here, we consider all $c$ and show that for ${a \in \{1,2,3\}}$, there are infinite families of lacunary series. However, for $a \geq 4$, we show that there are finitely many triples $(a,b,c)$ such that $F_{a,b,c}(z)$ is lacunary. In particular, if $a \geq 4$, $b \geq 7$, and $c \geq 2$, then $F_{a,b,c}(z)$ is not lacunary. Underlying this result is the proof the $t$-core partition conjecture proved by Granville and Ono \cite{Granville}.

\end{abstract}


\section{Introduction and Statement of Results}
A series $\displaystyle{\sum_{n=0}^{\infty} a(n)q^n}$ is said to be \textit{lacunary} if ``almost all'' of its coefficients are zero; that is, if we have 
$$ \displaystyle{\lim_{x \rightarrow \infty} \frac{ \# \{ 0 \leq n < x : a(n)=0 \} }{x}=1}.
$$ 
Two examples of generating functions which are lacunary come from identities due to Euler and Jacobi:

\begin{equation}\label{Euler}
\prod_{n = 1}^{\infty} \left( 1-q^{n} \right) = \sum_{m= -\infty}^{\infty} (-1)^{m}q^{\frac{3m^{2}+m}{2}}
\end{equation}

\begin{equation}\label{Jacobi1}
\prod_{n = 1}^\infty(1-q^{n})^3 = \sum_{m =-\infty}^{\infty} (-1)^{m}(2m+1)q^{\frac{m^2+m}{2}}.
\end{equation}

These combinatorial identities all have partition-theoretic interpretations. For example, Euler's identity (\ref{Euler}) gives a recurrence for calculating $p(n)$, the number of partitions of $n$. Further, they are all related to modular forms by the Dedekind $\eta$-function,
\begin{equation} \label{eta}
\eta(z) \coloneqq q^{\frac{1}{24}} \prod_{n = 1}^\infty (1-q^{n})
\end{equation}
where $q=e^{2\pi i z}$ and $\text{Im}(z)>0$. It is well-known that $\eta(24z)$ is a weight $\frac{1}{2}$ modular cusp form on $\Gamma_{0}(576)$ with Nebentypus character $\left( \frac{12}{n} \right)$. For background on modular forms, see \cite{Ono}.

In light of (\ref{Euler}) and (\ref{Jacobi1}), it is natural to investigate the lacunarity of $$\displaystyle{f_{r}(z) \coloneqq \prod_{n = 1}^\infty(1-q^{n})^{r}= \sum_{m=0}^{\infty} \tau_{r}(m)q^{m}}.$$
The series $f_{r}(z)$ is closely related to modular forms via the Dedekind $\eta$-funtion. Indeed, $\eta(24z)^{r}=q^{r}f_{r}(24z)$ is a modular cusp form of weight $\frac{r}{2}$ on $\Gamma_{0}(576)$. Via the connection of $f_{r}(z)$ to modular forms, Serre \cite{Serre} investigated the lacunarity of $\eta(24z)^{r}$ and thus $f_{r}(z)$: he proved that, if $r$ is even, then $f_{r}(z)$ is lacunary if and only if $r \in \{ 2,4,6,8,10,14,26 \}$. 

\begin{remark} If $r$ is odd, the lacunarity of $f_{r}(z)$ is not as well-understood since, in this case, $\eta(24z)^r$ has half-integral weight, and so the methods developed by Serre are not applicable. 
Of course, it is evident from $(\ref{Euler})$ and $(\ref{Jacobi1})$ that both $f_1(z)$ and $f_3(z)$ are lacunary.
\end{remark}

At first glance, it is not obvious that the series $f_{r}(z)$ is always a relative of the partition generating function
\[f_{-1}(z) = \sum_{m = 0}^\infty p(m) q^m = \prod_{n = 1}^\infty \frac{1}{1 - q^n}.\] However, work by Nekrasov and Okounkov \cite{NO} provides the key connection via an extension of this generating function. Recall that for a partition $\lambda$ of an integer $n \geq 1$, we can label each box $u$ of its Ferrers Diagram with a positive integer called the \textit{hooklength}, the number of boxes $v$ such that $v = u$, $v$ lies in the same column and below $u$, or $v$ lies in the same row and to the right of $u$. For example, when $n = 11$, we have the following Ferrers diagram for the partition $\lambda'=6+4+1$: 

\begin{figure}[h] \label{Ferrers}
\[\young(865421,5321,1)\]
\caption{The Ferrers diagram of $\lambda' = 6 + 4 + 1$}
\end{figure}
We define the multiset of all hook lengths of $\lambda$ to be $\mathcal{H}(\lambda)$. For $\lambda'$, we have $\mathcal{H}(\lambda') = \{1, 1, 1, 2, 2, 3, 4, 5, 5, 6, 8\}$. Nekrasov and Okounkov \cite{NO} generalized a result of Macdonald \cite{MacDonald} to create a partition-theoretic generating function. They showed that, for any $b \in \mathbb{C}$, one has
$$f_{b-1}(z)= \prod_{n = 1}^\infty(1-q^{n})^{b-1} =\sum_{\lambda \in \mathcal{P}}q^{|\lambda|} \prod_{h \in \mathcal{H}(\lambda)} \left( 1- \frac{b}{h^{2}} \right) .$$

Given an integer $a \geq 1$, we consider the subset $\mathcal{H}_{a}(\lambda) \subseteq \mathcal{H}(\lambda)$ 
defined by 
\begin{equation}\label{Ha}
\mathcal{H}_a(\lambda) \coloneqq \{ h\, \colon \, h \in \mathcal{H}(\lambda), h \equiv 0 \pmod{a} \}.
\end{equation}
A partition $\lambda$ is said to be an \textit{a-core} if $\mathcal{H}_{a}(\lambda)= \emptyset$. For $\lambda'$, we have that $\mathcal{H}_{2}(\lambda')= \{ 2,2,4,6,8 \}$ and that $\lambda'$ is an \textit{$a$-core} for $a=7$ and all $a \geq 9$. Han \cite{Han} generalized Nekrasov and Okounkov's formula to incorporate this set $\mathcal{H}_{a}(\lambda)$: 
\begin{equation} 
\displaystyle\sum_{\lambda \in \mathcal{P}} q^{|\lambda|} \prod_{h \in \mathcal{H}_a(\lambda)} \left( y-\frac{aby}{h^2}\right) = \prod_{n = 1}^\infty \frac{(1-q^{an})^a(1 - (yq^a)^n)^{b - a}}{ (1 - q)^n},
\end{equation}
where $b,y \in \mathbb{C}$. When $y = a = 1$, one recovers the identity of Nekrasov and Okounkov. We are interested in the case where $y = q^{a(c - 1)}$ for $c \geq 1$. Hence, for $a,b,c \geq 1$ we define
\begin{equation}\label{Fabc}
F_{a,b,c}(z) \coloneqq  \frac{\eta(24az)^a \eta(24acz)^{b-a}}{\eta(24z)}= q^{r} \prod_{n = 1}^\infty \frac{(1-q^{24an})^{a} (1-q^{24acn})^{b-a}}{1-q^{24n}},
\end{equation}
where $r \coloneqq abc+a^{2}-a^{2}c-1$.

\begin{remark}
In Lemma $\ref{Fabcmod}$, we show that $F_{a,b,c}(z)$ is a weakly holomorphic modular form of weight $\frac{b-1}{2}$ whenever $b$ is odd. From now on, we assume $a,b,c \geq 1$ are all integers and $b$ is odd. 
\end{remark}

Note that this is a generalization of Clader, Kemper and Wage's \cite{REU} work for which they took $c = 1$, corresponding to $y = 1$. In view of Serre's work and Han's partition-theoretic interpretation, they investigated the lacunarity of $F_{a,b,1}(z)$ and  showed that there are finitely many pairs $(a,b)$ for which $F_{a,b,1}(z)$ is lacunary. The complete list they gave is replicated in the table below:
\begin{center}
\begin{tabular}{|c | c|} \hline
$a$ & $b$ such that $F_{a,b,1}$ is lacunary \\ \hline 
$1$ & $\{ 3,5,7,9,11,15,27 \}$ \\
$2$ & $\{ 3,5,7 \}$ \\
$4$ & $ \{ 5,7 \}$ \\
$5$ & $ \{ 7,11 \}$ \\
$6$ & $ \emptyset $ \\
$7$ & $ \{9,15 \}$ \\
\hline
\end{tabular}
\end{center}

We now focus our attention to the case where $c \geq 1$. It turns out that when we allow $c$ to vary, there are infinitely many lacunary $F_{a,b,c}(z)$. In fact, our first theorem shows that for $a \in \{1,2,3\}$,  there exist infinite families of triples $(a,b,c)$ such that  $F_{a,b,c}(z)$ is lacunary. 
\begin{thm} \label{infinite}
For $F_{a,b,c}(z)$ as defined above, the following are true: \\
\noindent (1) If $b = 3$ and $a \in \{1,2,3\}$, then $F_{a,b,c}(z)$ is lacunary for all $c$. \\
\noindent (2) If $b = 5$ and $a \in \{1,2\}$, then $F_{a,b,c}(z)$ is lacunary for all $c$.
\end{thm}

\begin{remark}
Note that this is not a complete classification of lacunary $F_{a,b,c}(z)$ for $a \in \{ 1,2,3\}$.
\end{remark}
However, if $a \geq 4$, then we obtain a different phenomenon. Namely, there are only finitely many triples $(a,b,c)$ such that $F_{a,b,c}(z)$ is lacunary. 

\begin{thm} \label{finite}
If $a \geq 4$, $b \geq 1$ is odd, and $c \geq 2$, then $F_{a,b,c}(z)$ is not lacunary apart from three possible exceptions where $(a,b,c) \in \{(4,5,3), (4,5,5), (4,5,11)\}$.
\end{thm}

To obtain this result, we make use of the theory of modular forms with complex multiplication, Granville and Ono's \cite{Granville} proof of the $t$-core partition conjecture, and the P\'olya-Vinogradov Inequality. 

We begin by recalling some preliminaries in \S \ref{prelim}. In \S  \ref{proofs}, we prove Theorem \ref{infinite} and Theorem \ref{finite}. Finally, in \S \ref{examples} we conjecture that $F_{4,5,3}(z)$, $F_{4,5,5}(z)$, and $F_{4,5,11}(z)$ are lacunary and discuss a method by which one can show this.

\section{Preliminaries} \label{prelim}

In this section we recall the modularity properties of $\eta$-quotients as well as various conditions on modular forms for lacunarity. 
\subsection{$\eta$-quotients} \label{etaquot}
Let $M_k(\Gamma_0(N), \chi)$ (resp. $S_k(\Gamma_0(N), \chi))$ denote the complex vector space of holomorphic (resp. cuspidal) modular forms of weight $k$  with Nebentypus character $\chi$ on $\Gamma_0(N)$. An important example of a cuspidal modular form is the Dedekind $\eta$-function, defined in (\ref{eta}). Products of powers of the Dedekind $\eta$-function, $\eta$-quotients, are central to this paper. Formally, an $\eta$-quotient is any function $f$ that can be written as
\begin{equation}
f(z) = \prod_{\delta \mid N} \eta(\delta z)^{r_\delta}
\end{equation}
for some $N$ and each $r_{\delta} \in \mathbb{Z}$. Theorem 1.64 of \cite{Ono} gives conditions for when certain $\eta$-quotients are weakly holomorphic, which means that the poles are supported at cusps. In particular, if there is $N$ such that
\[ k \coloneqq \frac{1}{2} \sum_{\delta \mid N} r_{\delta} \in \mathbb{Z}, \]

\[\sum_{\delta \mid N} \delta r_\delta \equiv 0 \pmod {24},\] and \[\sum_{\delta \mid N} \frac{N}{\delta}r_{\delta} \equiv 0 \pmod {24},\] then $f(z)$ is a weakly holomorphic modular form of weight $k$ on $\Gamma_0(N)$ with Nebentypus character $\chi(d) \coloneqq \left(\frac{(-1)^k s}{d}\right)$, where $\displaystyle s \coloneqq  \prod_{\delta \mid N} \delta^{r_\delta}$. We can also compute the order of vanishing of $f(z)$ at a cusp $\frac{x}{y}$, for $y \mid N$ and $\gcd (x,y) = 1$ using the formula from Theorem 1.65 of \cite{Ono}: 
\begin{equation}\label{cuspeq}
\frac{N}{24}\sum_{\delta \mid N}\frac{\gcd (y, \delta)^2 r_\delta}{\gcd (y, \frac{N}{y})y \delta}.
\end{equation}
By considering the order of vanishing for all possible $y \mid N$, one can determine whether $f(z)$ is holomorphic or cuspidal. Applying this theory to our series $F_{a,b,c}(z)$, we have the following: 

\begin{lemma}\label{Fabcmod} Let $F_{a,b,c}(z)$ be as in (\ref{Fabc}). Then $F_{a,b,c}(z)$ is a weakly holomorphic modular form of weight $k=\frac{b-1}{2}$ on $\Gamma_{0}(576ac)$ with  Nebentypus character 
$$ \chi_{F}(d) = 
\begin{cases}
\left( \frac{(-1)^{k}ac}{d} \right) \text{ if } $a$ \text{ is even} \\
\left( \frac{(-1)^{k}a}{d} \right) \text{ if } $a$ \text{ is odd}. 
\end{cases}
$$   
Furthermore, $F_{a,b,c}(z)$ is holomorphic if and only if $b \geq a$ and cuspidal if and only if $b>a$.
\end{lemma}

\begin{remark}\label{leveloptimality} We choose to include a factor of $24$ in the definition of $F_{a,b,c}(z)$ to guarantee that the conditions mentioned above are satisfied for all triples $(a,b,c)$. This implies that $F_{a,b,c}(z)$ is a weakly holomorphic modular form on $\Gamma_0(N)$ where $N = 576ac$. In general, this $N$ is not necessarily optimal. The optimal $N$ may require altering the factor of 24 in the definition of $F_{a,b,c}(z)$ to some smaller $m$ satisfying $mr \equiv 0 \pmod{24}$ for $r$ defined in (\ref{Fabc}). The optimal level of this new modular form is then $N_{opt} \coloneqq \lcm(a,c)nm$ where the integer $n \geq 1$ is chosen minimally to satisfy $\frac{\lcm(a,c)}{ac} \cdot n(b-a) \equiv 0 \pmod{24}$.
\end{remark}

\subsection{Conditions on Modular Forms for Lacunarity} \label{modlac}
The lacunarity of certain modular forms is already known. Suppose first that $f(z)$ is a weight-one holomorphic modular form on a congruence subgroup. Deligne and Serre (Proposition 9.7 in \cite{DeligneSerre}) showed that $f(z)$ must be lacunary by expressing the coefficients in terms of the Artin $L$-function. Next, suppose that $f(z)$ is a weakly holomorphic modular form of  integral weight $k \geq 2$. If $f(z)$ has a pole at one of its cusps, then it cannot be lacunary. This can be shown using the Hardy-Littlewood circle method, which gives estimates for the magnitude of the Fourier coefficients of $f(z)$ as in Chapter 5 of \cite{Apostol}. Additionally, if $f(z)$ is holomorphic but not cuspidal, then it cannot be lacunary since $f(z)$ can be expressed as a linear combination of Eisenstein series together with a cusp form. Known bounds on the coefficients of such forms preclude them from being lacunary.

In light of these remarks, we may restrict our study of the lacunarity of $F_{a,b,c}(z)$ to the space of cusp forms of level $N$ and integral weight $k \geq 2$. Central to this study are modular forms with complex multiplication. Serre  \cite{Serre} proved that an element of $S_{k}(\Gamma_{0}(N),\chi)$ is lacunary if and only if it is expressible as a linear combination of forms with complex multiplication. We denote the space of cusp forms of weight $k$ with complex multiplication and Nebentypus character $\chi$ on $\Gamma_{0}(N)$ by $S^{CM}_k(\Gamma_{0}(N), \chi)$ and will sometimes call forms with complex multiplication CM forms. 

Recall that given any modular form $f(z) \in M_{k}(\Gamma_{0}(N), \chi)$, the action of the Hecke operator $T_{s}$ on $f(z)$ is defined by 
$$ f(z) | T_{s}= \displaystyle{\sum_{m=0}^{\infty}\left( \sum_{d \mid \gcd(s,m)} \chi(d)d^{k-1}a(sm/d^{2})\right)q^{m}},$$
where $\chi(m)=0$ if $\gcd(N,m) \neq 1$. When $s=p$ is prime, this definition reduces to 
$$f(z) | T_{p} = \displaystyle{\sum_{m=0}^{\infty}\left( a(pm)+\chi(p)p^{k-1}a(m/p) \right)q^{m}},$$
where we define $a(m/p) \coloneqq 0$ whenever $p \nmid m$.
We now state an extension of a well-known lemma by Serre \cite{Serre} which can be seen in \cite{REU}. It says that given a cusp form with complex multiplication, the action of certain Hecke operators is zero.  
\begin{lemma} \label{cmlem}
Let $g \in S^{CM}_k(\Gamma_{0}(N),\chi)$. If $s$ is an integer such that $\gcd(N,s)=1$ and there is no ideal of norm $s$ in the ring of integers of $\mathbb{Q}(\sqrt{-d})$ for any $d > 0$ with $d \mid N$, then $g | T_{s} =0$. 
\end{lemma}
\begin{remark}Note that for a given $N$, there always exists an $s$ satisfying the hypotheses of the previous lemma. This follows since a positive proportion of all integers are relatively prime to $N$ and, for a fixed $d$, $0 \% $ of integers are norms of ideals in the ring of integers of $\mathbb{Q}(\sqrt{-d})$.
\end{remark}

\section{Proofs of Theorems} \label{proofs}

\subsection*{Proof of Lemma \ref{Fabcmod}}

For $F_{a,b,c}(z)$, take $N = 576ac$ and $\delta_1 = 24a$, $\delta_2 = 24ac$, $\delta_3 = 24$ as seen in \S \ref{etaquot}. So $r_{\delta_1} = a$, $r_{\delta_2} = b - a$, and $r_{\delta_3} = -1$. Then  $F_{a,b,c}(z)$ satisfies the appropriate conditions to be a weakly holomorphic modular form with Nebentypus character $\chi_F(d) = \left( \frac{(-1)^k 24^{b - 1}a^b c^{b - a}}{d}\right)$. Properties of the Kronecker symbol give the desired $\chi_F$ for $a$ even and odd. 

The order of vanishing of a cusp $\frac{x}{y}$ obtained using (\ref{cuspeq}) is given by the expression
\begin{equation*}
\frac{\gcd(y, 24a)^2 a c + \gcd({y, 24ac)^2 (b - a) - \gcd(y, 24)^2ac}}{\gcd(y, \frac{576ac}{y})}.
\end{equation*}
\noindent As the numerator is nonnegative when $b \geq a$, $F_{a,b,c}(z)$ is holomorphic under this condition. 
To show the conditions of Lemma \ref{Fabcmod} precisely, we take $y = 24$, for which the numerator is nonnegative for $b \geq a$ and strictly positive for $b > a$. Therefore, $F_{a,b,c}(z)$ is only holomorphic for $b \geq a$ and cuspidal for $b > a$. \qed

\subsection*{Proof of Theorem \ref{infinite}} For (1), it suffices to recall from \S \ref{modlac} that weight-one holomorphic modular forms are lacunary. For (2), if $a = 1$, then $F_{1,5,c}(z) = \eta(24cz)^4$, which is lacunary due to Serre's \cite{Serre} classification of even $r$'s for which $\eta(z)^r$ is lacunary. If $a = 2$, we use (\ref{Jacobi1}) and an identity given in \cite{Ono} to express $F_{2,5,c}(z)$ as follows: 
\begin{align*} 
F_{2,5,c}(z) &= \frac{\eta(48z)^2}{\eta(24z)} \cdot \eta(48cz)^3 = \left(\sum_{n = 0} ^ \infty q^{3(2n + 1)^2} \right) \left( \sum_{m = 0}^\infty (-1)^m (2m + 1) q^{3c (2m+1)^2}\right) \\
&= \sum_{n = 0}^\infty \sum_{m = 0}^\infty (-1)^m(2m + 1)q^{3(2n + 1)^2 + 3c(2m + 1)^2}.
\end{align*}
As can be seen in \cite{Serre1}, binary quadratic forms represent 0\% of the positive integers, so $F_{2,5,c}(z)$ is lacunary for all $c$. 
 \qed

\subsection{Proof of Theorem \ref{finite}}
To prove this theorem, we first show that for a fixed pair $(a,c)$, there are only finitely many $b$ such that $F_{a,b,c}(z)$ is lacunary. Next, we show that there are finitely many pairs $(a,c)$ such that $F_{a,b,c}(z)$ is lacunary by bounding the product $ac$. This implies that there are finitely many triples $(a,b,c)$ such that $F_{a,b,c}(z)$ is lacunary. Finally, we review further steps for eliminating triples from being lacunary. In \S \ref{examples}, we discuss a method that one could use to complete this classification.

\begin{lemma}\label{finb} 
Fix $a \geq 4$ and $c \geq 1$. Then there exist finitely many odd $b \geq 1$ such that $F_{a,b,c}(z)$ is lacunary. 
\end{lemma}

\begin{proof}
Choose $s$ satisfying the conditions of Lemma \ref{cmlem}. As seen in \S \ref{prelim}, a result of Serre \cite{Serre} states that $F_{a,b,c}(z)$ lacunary if and only if it is expressible as a linear combination of CM forms. By Lemma \ref{cmlem}, $F_{a,b,c}(z)$ is not lacunary unless $F_{a,b,c}(z) | T_s = 0$. We show only finitely many odd $b$ satisfy this condition. 

Suppose for simplicity that $s = p$ is prime. This is possible since $s$ is square-free and $T_{\alpha}T_{\beta}=T_{\alpha \beta}$ whenever $\gcd(\alpha,\beta)=1$. Define the coefficients $A_{a,b,c}(m)$ by the expansion
\begin{equation} \label{DefFabc}
F_{a,b,c}(z) = q^{r} \prod_{n = 1}^\infty\frac{(1-q^{24an})^{a} (1-q^{24acn})^{b-a}}{1-q^{24n}}  = \sum_{m = 0}^\infty A_{a,b,c}(m)q^{24m+r}, 
\end{equation}
for $r$ as defined in (\ref{Fabc}). 
Therefore, acting on $F_{a,b,c}(z)$ with $T_p$, we obtain
\begin{equation}\label{DefineAabc}
F_{a,b,c}(z)| T_p = \sum_{\substack{m \geq 0 \\ p \mid 24m + r}}A_{a,b,c}(m) q^{\frac{24 m + r}{p}} + \chi_F(p) p ^\frac{b - 3}{2} \sum_{m \geq 0} A_{a,b,c}(m)q^{p(24m + r)}.
\end{equation} 
Choose $m_0 \geq 0$ to be the minimal $m$ so that $p \mid 24 m + r$. We must have $0 \leq m_{0} \leq p-1$. We show that this term cannot appear in the second sum, so $A_{a,b,c}(m_0)$ is the coefficient of $q^{\frac{24 m_0 + r}{p}}$ in the Fourier expansion of $F_{a,b,c}(z) | T_p$ \footnote{Note that when $s$ is not prime, there is $m_0 \leq s - 1$ such that $A_{a,b,c}(m_0)$ is a coefficient in the expansion of $F_{a,b,c}(z) \mid T_s$. However this $m_0$ is not necessarily the minimal $m$ such that $s \mid 24 m + r.$}. First note that for $m \geq 1$,
$$ p(24m + r) > \frac{24 m_0 + r}{p},$$ by the inequality $m_0 \leq p -1$. 
Next, looking at the first term of the second sum, we show that $\frac{24 m_0 + r}{p} \neq pr$. If this were true, using the inequality $0 \leq m_0 \leq p - 1$ again, one obtains: \[24(p - 1) \geq 24 m_0 = r(p^2 - 1),\] which gives \[24 \geq (p + 1)r = (p + 1)(abc + a^2 - a^2c - 1).\] This inequality can only hold if $p \leq 23$. Each prime $p \leq 23$ yields a finite set of triples $(a,b,c)$ for which the inequality holds. We reach a contradiction for each triple by calculating $m_0$ and checking equality of the first exponents of each of the two sums.

Therefore, it suffices to show $A_{a,b,c}(m_{0}) = 0$ for finitely many $b$. Notice from (\ref{DefFabc}) that \[F_{a,b,c}(z) = q^{r}\prod_{n = 1}^\infty (1 + q^{24n} + \cdots + q^{24(a - 1)})(1 - q^{24an})^{a - 1}(1 - q^{24acn})^{b - a} = \sum_{m = 0}^\infty A_{a,b,c}(m)q^{24m+r}.\] Since $a$ and $c$ are fixed, the coefficient $A_{a,b,c}(m)$ is a polynomial in $b$ with degree at most $\left\lfloor \frac{m}{ac} \right\rfloor$. Let's observe this phenomenon for $a = 4$ and $c = 1$: 
$$ F_{4,b,1}(z) = q^r \prod_{n=1}^{\infty} \frac{ (1-q^{96n})^{b}}{1-q^{n}} = q^r \prod_{n=1}^{\infty} \left( 1+q^{24n}+q^{48n}+q^{96n} \right) \left( 1-q^{96n} \right)^{b-1}.$$ By considering the possible ways that terms in this product could multiply to $q^{96 + r}$, we deduce that 
$$A_{4,b,1}(4)= -\binom{b-1}{1}+1+1+1+1=-b+5.$$ 
Hence $A_{4,b,1}(4)$ is a polynomial in $b$ of degree at most $\left\lfloor \frac{4}{4 \cdot 1} \right\rfloor = 1$.

We now show that for $a \geq 4$, $A_{a,b,c}(m)$ is a nonzero polynomial in $b$ for all $m$. Therefore $A_{a,b,c}(m_0) = 0$ for finitely many $b$.  Plugging $b=a$ into (\ref{DefFabc}), we have: 
\begin{equation}
\prod_{n = 1}^\infty \frac{(1-q^{24an})^{a}}{(1-q^{24n})}= \sum_{m = 0}^\infty A_{a,a,c}(m)q^{24m}.
\end{equation}
\noindent We show $A_{a,a,c}(m) \neq 0$ for $m \geq 0$. Note that from Corollary 1.9 in \cite{Han} we have the following:
\begin{align*}
\sum_{m = 0}^\infty A_{a,a,c}(m)q^{24m} = \prod_{n = 1}^\infty \frac{(1-q^{24an})^{a}}{1-q^{24n}} &= \sum_{\lambda \text{ is an $a$-core }}q^{24 |\lambda|} \\
&= \sum_{m = 0}^\infty \# \{ \lambda \colon | \lambda |=m \text{ and } \lambda \text{ is an $a$-core } \}q^{24m}.
\end{align*}
\noindent It then follows from Granville and Ono's \cite{Granville} proof of the $t$-core partition conjecture  that $\# \{ \lambda \colon | \lambda | =m \text{ and } \lambda \text{ is an $a$-core } \} \geq 1$, since we are assuming $a \geq 4$. Thus, $A_{a,b,c}(m)$ is a nonzero polynomial in $b$ for all $m \geq 0$. 

\end{proof}
The following lemma is a direct generalization of Theorem 1.1 in \cite{REU}. In order to prove this lemma, we recall the P\'olya-Vinogradov Inequality which states that given a nontrivial Dirichlet character $\chi$ with modulus $m$, the following holds: 
\begin{equation}
\big| \sum_{x=1}^{h}\chi(x) \big| \leq 2 \sqrt{m} \log m.
\end{equation}

 \begin{lemma} \label{finitelymanytriples}
If $a \geq 4$ there exists finitely many triples $(a,b,c)$ so that $F_{a,b,c}(z)$ is lacunary.
 \end{lemma}
 
\begin{proof}
Fix $a \geq 4$  and $c \geq 1$. Let $a'$ be the square-free part of $576ac$, the level of $F_{a,b,c}(z)$. We will show that there exists an $s \in \mathbb{Z}$ satisfying the conditions of Lemma \ref{cmlem}. In order to do this,  it suffices to show that there exists an $s \in \mathbb{Z}$ such that $ \left( \frac{-1}{s} \right) = -1$ and $\left(  \frac{-p}{s} \right)= -1$ for all primes $p \mid a'$. These conditions imply that $\left( \frac{-d}{s} \right)= -1$ for all positive integers $d$ dividing the level of $F_{a,b,c}(z)$ by properties of the Kronecker symbol. Recall that a prime $q$ is inert in the ring of integers of $\mathbb{Q}(\sqrt{-d})$ if $\left( \frac{-d}{q} \right) = -1$. By considering the prime factorization of $s$ and the multiplicativity of the Kronecker symbol, there does not exist any ideal of norm $s$ in the ring of integers of $\mathbb{Q}(\sqrt{-d})$ if $\left( \frac{-d}{s} \right)= -1$. 

For each pair $(a,c)$, there exists an $s$ satisfying the conditions of Lemma \ref{cmlem}. We show that for all but finitely many pairs, there is an $s < ac$. By (\ref{finb}), the set of odd $b$ for which $F_{a,b,c}(z)$ is lacunary is a subset of the roots of the nonzero polynomials $A_{a,b,c}(0), \ldots, A_{a,b,c}(s-1)$. All of these polynomials have degree less than or equal to $\big\lfloor \frac{s}{ac} \big\rfloor$, so no $b$ can exist for all but finitely many pairs $(a,c)$. 

We first assume that $a'=6ac$. We show that there is a $\kappa$ such that whenever $a' = 6ac \geq \kappa$ there exists an $s < ac$ satisfying the desired conditions. It suffices to find an $s \equiv 23 \pmod{24}$ such that $\left( \frac{-p}{s} \right) =-1$ for all primes $p \mid ac$ $\left( \text{note that if $s \equiv 23 \pmod{24}$, $\left( \frac{-2}{s} \right) = \left( \frac{-3}{s}\right) =-1$} \right)$. Write $ac=p_{1}p_{2} \cdots p_{m}$, where the $p_{i}$ are distinct primes not equal to $2$ or $3$. Consider the following linear combination of Dirichlet characters: \begin{equation*}
g_{ac}(s) \coloneqq \frac{1}{2^{m}} \cdot \displaystyle{\sum_{d \mid ac} \mu(d)\psi_{d}(s)},
\end{equation*} where $\psi_{d}(s) \coloneqq \displaystyle{\prod_{p \mid d}} \left( \frac{-p}{s}\right)$ and by convention $\psi_1(s) = 1$. By inducting on the number of prime divisors of $ac$ and noticing that if $q$ is a prime not dividing $ac$ then $g_{ac \cdot q}(s) = \frac{\left(1- \left( \frac{-q}{s} \right) \right)g_{ac}(s)}{2}$, one can show that 
$$ g_{ac}(s)= 
\begin{cases}
1 \text{ if } \left( \frac{-p}{s} \right)=-1 \text{ for all } p \mid ac \\
0 \text{ otherwise.}
\end{cases}$$
Thus, it is sufficient to find an $s$ satisfying $g_{ac}(s)=1$ and $s \equiv 23 \pmod{24}$. We will do this by using the P\'olya-Vinogradov Inequality to show that, if $ac$ is sufficiently large, then 
\begin{equation}\label{sac} 
\sum_{\substack{s < ac \\ s \equiv 23 (24)}} \sum_{d \mid a} \mu(d)\psi_{d}(s) >0.
\end{equation}
Let $D$ be the set of Dirichlet characters modulo $24$. Since $\displaystyle{\sum_{\chi \in D} \chi(1)=8}$, we have:
\begin{align*}
\sum_{\substack{s < ac \\ s \equiv 23 (24)}} \sum_{d \mid ac} \mu(d)\psi_{d}(s) &= \sum_{\substack{s < ac \\ s \equiv 23 (24)}} 1 + \sum_{\substack{d \mid ac \\ d>1}} \mu(d) \sum_{\substack{s <ac \\ s \equiv 23 (24)}} \psi_{d}(s) \\
&= \sum_{\substack{s<ac \\ s\equiv 23 (24)}} 1 + \sum_{\substack{d \mid ac \\ d>1}} \mu(d) \sum_{s <ac} \psi_{d}(s) \sum_{\chi \in D} \frac{\chi(-s)}{8} \\
&\geq \frac{ac}{24}-1+ \sum_{\substack{d \mid ac \\ d>1}} \frac{1}{8} \mu(d) \psi_{d}(-1) \sum_{\chi \in D} \sum_{s < ac} \psi_{d}(-s) \chi(-s). 
\end{align*}
Note that $\psi_d \circ \chi$ is a nontrivial Dirichlet character modulo $24d$. So, by applying the P\'olya-Vinogradov Inequality to the innermost sum: 
\begin{align*} 
\sum_{\substack{s <ac \\ s \equiv 23(24)}} \sum_{d \mid ac} \mu(d) \psi_{d}(s) &\geq \frac{ac}{24} - 1 - \sum_{\substack{d \mid ac \\ d >1}} 2 \sqrt{24d} \log(24d) \\
&\geq \frac{ac}{24} -1 -2^{m+1} \sqrt{24ac} \log(24ac).
\end{align*}

If $m \geq 12$, then $\frac{ac}{24} -1 -2^{m+1} \sqrt{24ac} \log(24ac) >0$. Hence, there exists an integer $s <ac$ satisfying the conditions of Lemma \ref{cmlem}. So in this case, the set of $b$ for which $F_{a,b,c}(z)$ is lacunary is empty. For each $m <12$, the inequality (\ref{sac}) can only fail if $ac < \kappa$, where $\kappa = 5 \cdot 7 \cdot \ldots \cdot 43 \approx 2.18 \times 10^{15}$. So if $6ac$ is square-free, there are only finitely many pairs $(a,c)$ for which there could exist odd $b$ where $F_{a,b,c}(z)$ is lacunary. 

Now we turn to the case where $6ac$ is not square-free. If $\frac{a'}{6}$ has a corresponding $s \leq \frac{a'}{6} < ac,$ then $F_{a,b,c}(z)$ is not lacunary for any choice of $b$. If $\frac{a'}{6}$ does not have an $s \leq \frac{a'}{6},$ then $F_{a,b,c}(z)$ can only be lacunary for pairs $(a,c)$ such that $s \leq ac$ where $s$ is chosen minimally. By the above, there are only finitely many possible $a',$ each of which yields finitely many pairs $(a,c)$ for which there could exist odd $b$ such that $F_{a,b,c}(z)$ is lacunary. It follows from Lemma \ref{finb} that there can be only finitely many pairs $(a,b,c)$ for which $F_{a,b,c}(z)$ is lacunary. 
\end{proof}

We now prove Theorem \ref{finite}. 

\begin{proof}
We demonstrate the process of showing that the only possible triples $(a,b,c)$ for which $F_{a,b,c}(z)$ could be lacunary are $(4,5,3)$, $(4,5,5)$, and $(4,5,11)$. We have the following algorithm: \\

\noindent (1) From Lemma \ref{finitelymanytriples} we know that there are finitely many triples $(a,b,c)$ such that $F_{a,b,c}(z)$ is lacunary. Further, all such triples have the property that $a' \leq 6\kappa$ where $a'$ and $\kappa$ are defined above. For each $a' \leq 6\kappa$, we check whether there exists an $s \equiv 23 \pmod{24}$ with $s < \frac{a'}{6}$ and $g_{a'/6}(s) = 1$. 
For all $a'$ with such an $s$, the possible pairs $(a,c)$ corresponding to $a'$ yield no odd $b$ for which $F_{a,b,c}(z)$ is lacunary. For all $a'$ without such an $s,$ there are finitely many pairs $(a,c)$ such that $F_{a,b,c}(z)$ could be lacunary. Let $S_{ac}$ be the set of all such pairs.  \\

\noindent (2) For all $(a,c) \in S_{ac}$, choose $s$ minimally and use Lagrange interpolation to construct the polynomials $A_{a,b,c}(ac), \ldots ,A_{a,b,c}(s-1)$ \footnote{Note that we do not need to consider the polynomials $A_{a,b,c}(0), \ldots, A_{a,b,c}(ac-1) $ since these polynomials are nonzero constant polynomials and hence have no roots.}. Now find all of the odd positive integer roots of these polynomials. By repeating this processes for all $(a,c)$ with corresponding $s$, we have a set of all possible $(a,b,c)$. Denote this set by $S_{abc}$. \\

\noindent (3) For all $(a,b,c) \in S_{abc}$, choose prime $p$ minimally satisfying $p \equiv 23 \pmod{24}$ and $g_{a'/6}(p)=1$. Recall the definition of $m_{0}$ in Lemma \ref{finb}. Find $m_{0} \in \mathbb{Z}/p\mathbb{Z}$ such that $24m_{0} \equiv -r \pmod{p}$. If $A_{a,b,c}(m_{0}) \neq 0$ then $F_{a,b,c}(z)$ cannot be lacunary. Denote the set of remaining triples $S_{abc}'$. \\

\noindent (4) For all $(a,b,c) \in S_{abc}',$ find $p' > p$ satisfying $p' = 23 \pmod{24}$ and $g_{a'/6}(p') = 1$ and repeat step 3. This should be repeated multiple times to eliminate further candidates. \\

Following steps (1)-(4), the only remaining triples $(a,b,c)$ such that $F_{a,b,c}(z)$ could be lacunary are $(4,5,3)$, $(4,5,5)$, and $(4,5,11)$.

\end{proof}

\section{Discussion of Theorem \ref{finite}}\label{examples}
Due to our results in \S \ref{proofs}, we conjecture that the series $F_{4,5,3}(z), F_{4,5,5}(z)$, and $F_{4,5,11}(z)$ are lacunary. In theory, it is not difficult to prove this conjecture. Namely, one has to systematically compute all of the weight $2$ modular forms with complex multiplication on a suitable level. Our calculations reveal that if this conjecture is true, then these $F_{a,b,c}(z)$ will be linear combinations of CM forms corresponding to fields in the table given below \footnote{There are fewer modular forms on lower level; hence, one should consider the optimal level as discussed in the remark in \S \ref{etaquot}.}. We have been unable to resolve this issue due to limits on our computational power.

\begin{center}
\begin{table}[h]
\begin{tabular}{|c | c| c |} \hline
 & Optimal level of $F_{a,b,c}(z)$  &  CM fields of $F_{a,b,c}(z)$  \\ \hline 
$F_{4,5,3}(z)$ & $2304$ & $\mathbb{Q}(\sqrt{-1})$, $\mathbb{Q}(\sqrt{-2})$, and $\mathbb{Q}(\sqrt{-3})$ \\ \hline
$F_{4,5,5}(z)$ & $576 \cdot 4 \cdot 5$ & $\mathbb{Q}(\sqrt{-1})$, $\mathbb{Q}(\sqrt{-2})$, $\mathbb{Q}(\sqrt{-3})$, and $\mathbb{Q}(\sqrt{-5})$ \\ \hline
$F_{4,5,11}(z)$ & $576 \cdot 4 \cdot 11$ & $\mathbb{Q}(\sqrt{-1})$, $\mathbb{Q}(\sqrt{-2})$, $\mathbb{Q}(\sqrt{-3})$, and $\mathbb{Q}(\sqrt{-11})$ \\ \hline

\end{tabular}
\end{table}
\end{center}

\section{Acknowledgements}
The authors wish to thank Professor Ken Ono and Professor Larry Rolen for their invaluable guidance and suggestions. They would also like to thank Emory University, the Asa Griggs Candler Fund, and NSF grant DMS-1557960.
 
\bibliographystyle{abbrv}
\bibliography{biblio}

\end{document}